\documentclass[11pt]{amsart}
\usepackage[all]{xy}
\usepackage{amsmath}
\usepackage{amsthm}
\usepackage{amsfonts}
\usepackage{amssymb}
\usepackage{amscd}
\usepackage{epic}
\usepackage{eepic}

\usepackage{graphicx}

\usepackage{verbatim}
\newtheorem{theorem}{Theorem}[section]
\newtheorem*{theorem*}{Theorem}
\newtheorem{proposition}[theorem]{Proposition}
\newtheorem{lemma}[theorem]{Lemma}

 \theoremstyle{definition}

\newtheorem{dfn}[theorem]{Definition}

\newcommand{\Z}{\mathbb{Z}}

\newcommand{\ki}{\mathfrak{i}}
\newcommand{\kj}{\mathfrak{j}}

\newcommand{\gam}{\gamma}
\newcommand{\CH}{\operatorname{CH}}
\newcommand{\Pt}{\mathcal{P}_2}
\newcommand{\sig}{\sigma}
\newcommand{\lra}{\longrightarrow}
\newcommand{\ra}{\rightarrow}
\newcommand{\ol}{\overline}
\newcommand{\lam}{\lambda}
\newcommand{\bs}{\backslash}
\newcommand{\Lam}{\Lambda}
\newcommand{\abs}[1]{\lvert#1\rvert}

\newcommand{\codim}{\operatorname{codim}}
\newcommand{\cdim}{\operatorname{cdim}}
\newcommand{\al}{\alpha}
\newcommand{\ds}{\oplus}
\newcommand{\mult}{\operatorname{mult}}
\newcommand{\de}{\delta}
\newcommand{\lt}{\leadsto}
\newcommand{\im}{\operatorname{im}}
\newcommand{\trans}{\operatorname{transpose}}
\newcommand{\Spec}{\operatorname{Spec}}

\renewcommand{\H}{\mathbb{H}}
\newcommand{\rat}{\dashrightarrow}
\renewcommand{\phi}{\varphi}
\newcommand{\pt}{\tilde{\phi}}
\newcommand{\Ch}{\operatorname{Ch}}

\newcommand{\subs}[2]{\genfrac{}{}{0pt}{}{#1}{#2}}

\title{Incompressibility of orthogonal Grassmannians of rank $2$}
\author{Bryant G. Mathews}
\date{}
\address{Department of Mathematics, University of California,
        Los Angeles, CA 90095-1555} \email {bmathews@math.ucla.edu}

\begin{document}
\begin{abstract}
For a nondegenerate quadratic form $\phi$ on a vector space $V$ of dimension $2n+1$, let $X_d$ be the variety of $d$-dimensional totally isotropic subspaces of $V$.  We give a sufficient condition for $X_2$ to be $2$-incompressible, generalizing in a natural way the known sufficient conditions for $X_1$ and $X_n$.  Key ingredients in the proof include the Chernousov-Merkurjev method of motivic decomposition as well as Pragacz and Ratajski's characterization of the Chow ring of $(X_2)_E$, where $E$ is a field extension splitting $\phi$.
\end{abstract}
\maketitle

\section{Preliminaries}

Before stating our results in the next section, we begin by recalling the notions of canonical $p$-dimension, $p$-incompressibility, and higher Witt index.

Let $X$ be a scheme over a field $F$, and let $p$ be a prime or zero.  A field extension $K$ of $F$ is called a \emph{splitting field of $X$} (or is said to \emph{split $X$}) if $X(K)\neq\emptyset$.  A splitting field $K$ is called \emph{$p$-generic} if, for any splitting field $L$ of $X$, there is an $F$-place $K\rightharpoonup L'$ for some finite extension $L'/L$ of degree prime to $p$.  In particular, $K$ is $0$-generic if for any splitting field $L$ there is an $F$-place $K\rightharpoonup L$.

The canonical $p$-dimension of a scheme $X$ over $F$ was originally defined \cite{BR05,KM06} as the minimal transcendence degree of a $p$-generic splitting field $K$ of $X$.  When $X$ is a smooth complete variety, the original algebraic definition is equivalent to the following geometric one \cite{KM06,M}.
\begin{dfn}\label{cddef}
Let $X$ be a smooth complete variety over $F$.  The \emph{canonical $p$-dimension $\cdim_p(X)$ of $X$} is the minimal dimension of the image of a morphism $X'\ra X$, where $X'$ is a variety over $F$ admitting a dominant morphism $X'\ra X$ with $F(X')/F(X)$ finite of degree prime to $p$.  The canonical $0$-dimension of $X$ is thus the minimal dimension of the image of a rational morphism $X\rat X$.
\end{dfn}
In the case $p=0$, we will drop the $p$ and speak simply of \emph{generic} splitting fields and canonical \emph{dimension} $\cdim(X)$.

For a third definition of canonical $p$-dimension as the essential $p$-dimension of the detection functor of a scheme $X$, we refer the reader to Merkurjev's comprehensive exposition \cite{M} of essential dimension.

For a smooth complete variety $X$, the inequalities $$\cdim_p(X)\leq \cdim (X)\leq \dim (X)$$ are clear from Definition \ref{cddef}.  Note also that if $X$ has a rational point, then $\cdim(X)=0$ (though the converse is not true).

\begin{dfn}
When a smooth complete variety $X$ has canonical $p$-dimension as large as possible, namely $\cdim_p(X)=\dim(X)$, we say that $X$ is \emph{$p$-incompressible}.
\end{dfn}

It follows immediately that if $X$ is $p$-incompressible, it is also \emph{incompressible} (i.e. $0$-incompressible). \\

We next recall the definitions of absolute and relative higher Witt indices, introduced by Knebusch in \cite{K76}.  Our discussion follows \cite[\S 90]{EKM08}.  The \emph{Witt index $\ki_0(\phi)$} of a quadratic form $\phi$ is the number of copies of the hyperbolic plane $\H$ which appear in the Witt decomposition of $\phi$.  Now let $\phi$ be a nondegenerate quadratic form over a field $F$ and set $F_0:=F$ and $\phi_0:=\phi_{an}$, the anisotropic part of $\phi$.  We proceed to recursively define $F_k:=F_{k-1}(\phi_{k-1})$, $\phi_k:=\left(\phi_{F_k}\right)_{an}$ for $k=1, 2, \ldots$, stopping at $F_h$ such that $\dim \phi_h\leq 1$.
\begin{dfn}
For $k\in\{0, 1, \ldots, h\}$, the \emph{$k$-th absolute higher Witt index $\kj_k(\phi)$ of $\phi$} is defined to be $\ki_0(\phi_{F_k})$.  For $k\in\{1,2,\ldots,h\}$, the \emph{$k$-th relative higher Witt index $\ki_k(\phi)$ of $\phi$} is defined to be the difference $$\ki_k(\phi):=\kj_k(\phi)-\kj_{k-1}(\phi).$$  The \emph{$0$-th relative higher Witt index of $\phi$} is the usual Witt index $\ki_0(\phi)$.
\end{dfn}

It follows from the definition that $$0\leq \kj_0(\phi)<\kj_1(\phi)<\cdots<\kj_h(\phi)=\left[(\dim\phi)/2\right].$$  Moreover, it can be shown that the set $\{\kj_0(\phi),\ldots,\kj_h(\phi)\}$ of absolute higher Witt indices of $\phi$ is equal to the set of all Witt indices $\ki_0(\phi_K)$ for $K$ an extension field of $F$. \\

\section{Introduction}
Let $\phi$ be a nondegenerate quadratic form on a vector space $V$ of dimension $2n+1$ over a field $F$.  Associated to $\phi$ there are smooth projective varieties $X_1,X_2,\linebreak[0]\ldots,\linebreak[0]X_n$, where $X_d$ is the variety of $d$-dimensional totally isotropic subspaces of $V$.  The variety $X_1$ is simply the projective quadric hypersurface associated to the quadratic form $\phi$.

We recall the following result proved in \cite{KM03} and also in \cite[Ch. XIV and \S 90]{EKM08}.

\begin{theorem}[Karpenko, Merkurjev]
If the quadric $X_1$ is anisotropic, then $$\cdim_2(X_1)=\cdim(X_1)=\dim(X_1)-\ki_1(\phi)+1.$$  In particular, $X_1$ is $2$-incompressible if and only if $\ki_1(\phi)=1$.
\end{theorem}

At the other extreme is the variety $X_n$ of maximal totally isotropic subspaces of $V$.  In \cite[Ch. XVI]{EKM08}, building on a result of Vishik from \cite{V05}, the canonical 2-dimension of $X_n$ is computed in terms of the $J$-invariant of $\phi$.  The following result is a corollary.

\begin{theorem}[Karpenko, Merkurjev]
If $\deg\CH(X_n)=2^n\Z$, then $X_n$ is $2$-incom\-pres\-sible.
\end{theorem}

To compute the canonical $2$-dimension of a general $X_d$ appears to be difficult because of the complexity of the Chow ring when $d\notin \{1,n\}$.  In this paper, we complete a small piece of this general program by determining a sufficient condition for the variety $X_2$ to be $2$-incompressible.  We assume everywhere that $n\geq 3$, the $n=2$ case having already been dealt with.

\begin{theorem}\label{main}
If $\deg\CH(X_2)=4\Z$ and $\ki_2(\phi)=1$, then $X_2$ is $2$-in\-com\-pres\-sible.  In particular, $$\cdim_2(X_2)=\cdim(X_2)=\dim(X_2)=4n-5.$$
\end{theorem}

This result concerning $X_2$ is a natural generalization of what is already known about $X_1$ and $X_n$.  To see this, note that $X_1$ being anisotropic means that $\deg\CH(X_1)\linebreak[0]=\linebreak[0]2\Z$.  Furthermore, \linebreak[0]$\deg\linebreak[0]\CH(X_n)\linebreak[0]=2^n\Z$ implies that $\kj_{n-1}(\phi)=n-1$, from which it immediately follows that $\ki_n(\phi)=1$.  One might then conjecture, for general $d$, that $$\deg\CH(X_d)=2^d\Z,\quad \ki_d(\phi)=1$$ are sufficient conditions for $X_d$ to be $2$-incompressible.

\section{Higher Witt indices}

In this section we collect two results concerning higher Witt indices which will be needed later.

\begin{proposition}\label{witt1}
If $\deg\CH(X_2)=4\Z$, then $\kj_1(\phi)=1$.
\end{proposition}

\begin{proof}
Let $K$ be a field of degree $2$ over $F$ such that the anisotropic part of $\phi$ has a rational point over $K$.  Then $\ki_0(\phi_K)>\ki_0(\phi)$.  By \cite[Prop. 25.1]{EKM08}, it follows that $\ki_0(\phi_K)\geq \kj_1(\phi)$.  If $\kj_1(\phi)\geq 2$ then so is $\ki_0(\phi_K)$, which implies that the variety $X_2$ has a rational point over $K$.  Since $K$ has degree $2$ over $F$, this contradicts the assumption.  Thus $\kj_1(\phi)=1$ and $\kj_0(\phi)=0$.
\end{proof}

From this proposition, we see that the hypothesis of our Theorem \ref{main} implies $$\kj_2(\phi)=\kj_1(\phi)+\ki_2(\phi)=1+1=2.$$

\begin{proposition}\label{witt2}
If $\kj_2(\phi)=2$, then $\ki_0(\phi_{F(X_2)})=2$.
\end{proposition}

In fact, one need only assume that {\em some} absolute higher Witt index of $\phi$ is equal to $2$ (possibly $\kj_0(\phi)$ or $\kj_1(\phi)$), but we don't need this generality for our purposes.

\begin{proof}
The variety $X_2$ has a rational point over $F(X_2)$, so $\ki_0(\phi_{F(X_2)})\geq 2$.

We prove that $\ki_0(\phi_{F(X_2)})\leq 2$ by contradiction.  If $\pt$ is the anisotropic part of $\phi_{F(\phi)}$, then $\phi_{F(\phi)}\simeq \H\perp\pt$, since our assumption $\kj_2(\phi)=2$ implies that $\kj_1(\phi)=1$.  We define two varieties over $F':=F(\phi)$.   Let $Y_1$ be the projective quadric corresponding to $\pt$, and let $Y_2$ be the variety of totally isotropic subspaces of dimension $2$ with respect to $\pt$.  Since $$\ki_0(\phi_{F'(Y_1)})= \ki_0(\phi_{F(\phi)(\pt)})=\kj_2(\phi)=2,$$ the variety ${X_2}$ has a rational point over $F'(Y_1)$.  If $\ki_0(\phi_{F(X_2)})\geq 3$, then $\ki_0\left(\pt_{F'\left((X_2)_{F'}\right)}\right)\linebreak[0]\geq\linebreak[0]2$, so $Y_2$ has a rational point over $F'\left((X_2)_{F'}\right)$.  We thus have two rational morphisms between varieties over $F'$:
$$Y_1\rat (X_2)_{F'}\rat Y_2.$$  By \cite[Lem. 6.1]{BRV}, since $X_2$ is smooth and $Y_2$ is complete, there exists a rational morphism from $Y_1$ to $Y_2$, i.e. $(Y_2)_{F'(Y_1)}$ has a rational point.  But this is a contradiction, since $\kj_1(\pt)=\kj_2(\phi)-1=1$.
\end{proof}

\section{Shapes and multiplicities}

A quadratic form $\phi$ is \emph{split} if $\ki_0(\phi)=\left[(\dim\phi)/2\right]$, the greatest possible value.  Given a quadratic form $\phi$ over $F$, we fix an extension field $E/F$ such that $\phi_E$ is split and define $\bar{\phi}:=\phi_E$ and $\bar{X}_d:=(X_d)_E$.  Let $\ol{\CH}(X_d)$, called the {\em reduced Chow group}, denote the image of the change of field homomorphism $\CH(X_d)\ra\CH(\bar{X}_d)$.  Elements of $\ol{\CH}(X_d)$ will be called {\em rational} cycles.

In this section we prove a technical lemma, based on a characterization given by Pragacz and Ratajski in \cite{PR96} of the Chow ring of the variety $\bar{X}_2$.

\begin{lemma}\label{tech}
If $\gam \in \CH^r(\bar{X}_2)$ with $r\in\{2n-3,2n-2\}$, then $2\gam$ is a rational cycle.
\end{lemma}

We begin by fixing notation and recalling some definitions.  Let $X_B:=\bar{X}_2,$ the variety of 2-dimensional isotropic subspaces of $V$ with respect to the {\em split} nondegenerate quadratic form $\bar{\phi}$.  Recall that $V$ has dimension $2n+1$.  Let $X_C$ denote the variety of 2-dimensional isotropic subspaces of a vector space $W$ of dimension $2n$ with respect to a nondegenerate alternating form $\psi$ on $W$.

The next group of definitions are adapted for our purposes from Macdonald's \cite{M79}.

\begin{dfn} A \emph{partition} is a finite, strictly decreasing sequence of positive integers
$$\lam^*=(\lam^*_1,\lam^*_2,\ldots,\lam^*_s).$$
\end{dfn}

The $*$ in our notation will take on the value $t$ for ``top'' or $b$ for ``bottom,'' depending on the role the partition plays in a shape, defined below.  The \emph{length} $l(\lam^*)$ of the partition above is just $s$, while the \emph{weight} of the partition is defined to be $|\lam^*|:=\sum_{k=1}^s \lam^*_k$.  The empty partition, denoted $\emptyset$, is the sequence with no terms.

Partitions are visualized as diagrams of boxes.  The \emph{diagram} $D^*_{\lam}$ of a partition $\lam^*$ has $\lam^*_k$ boxes in the $k$th row, beginning with the top row.  For example, the partition $(4,3,1)$ has diagram:
\begin{center}
\includegraphics[scale=0.75]{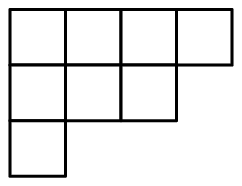}
\end{center}
Note that the length of a partition is just the number of rows in its diagram, while the weight is just the number of boxes.

A \emph{skew diagram} $D^*_\mu \backslash D^*_\lam$ is obtained by removing the boxes in the intersection of $D^*_\mu$ and $D^*_\lam$ from $D^*_\mu$.  For example, the skew diagram $(4,3,1)\backslash (2,1)$ is as follows:
\begin{center}
\includegraphics[scale=0.75]{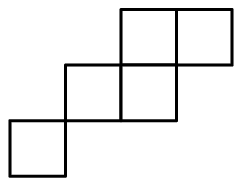}
\end{center}

The remaining definitions are adapted from \cite{PR96}.

\begin{dfn}
With $n$ fixed as above, a pair $\lam=(\lam^t//\lam^b)$ of partitions $\lam^t$ and $\lam^b$ is called a \emph{shape} if $\lam^t_1,\lam^b_1\leq n$, $l(\lam^t)\leq n-2$, $\l(\lam^b)\leq 2$, and $\lam^t_{n-2}>l(\lam^b)$.
\end{dfn}

The \emph{diagram} $D_\lam$ of a shape is drawn by stacking the diagram $D^t_\lam$ of the first partition on top of the diagram $D^b_\lam$ of the second, with a horizontal line in between.  For example, the shape $((4,2)//(3))$ has diagram:
\begin{center}
\includegraphics[scale=0.75]{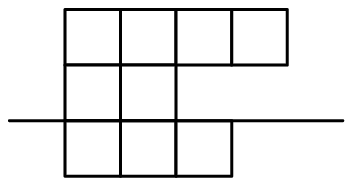}
\end{center}

The inequalities in the definition of a shape $\lam$ amount to imposing three requirements on its diagram $D_\lam$:
\begin{enumerate}
\item $D^t_\lam$ must fit into a rectangle with height $n-2$ and width $n$
\item $D^b_\lam$ must fit into a rectangle with height $2$ and width $n$
\item $D_\lam$ must contain the ``triangle'' of boxes lying on and above the diagonal beginning at the box in the diagram's lower-left corner and proceeding to the ``northeast''.
\end{enumerate}
We display once again the diagram of the shape $((4,2)//(3))$, this time shading in the triangle of boxes required by the third condition above:
\begin{center}
\includegraphics[scale=0.75]{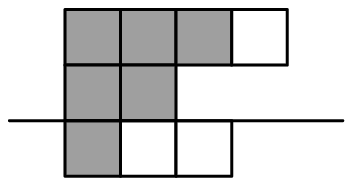}
\end{center}

The \emph{weight} of a shape $\lam=(\lam^t//\lam^b)$ is defined in terms of the weights of the partitions $\lam^t$, $\lam^b$ as $$\abs{\lam}:=\abs{\lam^t}+\abs{\lam^b}-{n-1 \choose 2}.$$
This definition is chosen so that the shape $\pi_0:=((n-2,n-3,\ldots,1)//\emptyset)$ of minimal weight will have weight $0$.

We refer to \cite{PR96} for the definitions of extremal and related components, $(\mu-\lam)$-boxes, and compatible shapes.

The set of all shapes, denoted $\Pt$, can be mapped bijectively onto bases for each of $\CH(X_B)$ and $\CH(X_C)$.  We name the maps $$\sig:\Pt \lra \CH(X_B)$$ $$\tau:\Pt \lra \CH(X_C)$$ and we call cycles in the images of the maps {\em basic} cycles.

For $i=0,1,2$, consider the shapes $\pi_i:= ((n-2+i, n-3, n-4, \ldots, 1)//\emptyset)$.  When $n=5$, their diagrams are as follows:
\begin{center}
\includegraphics[scale=0.75]{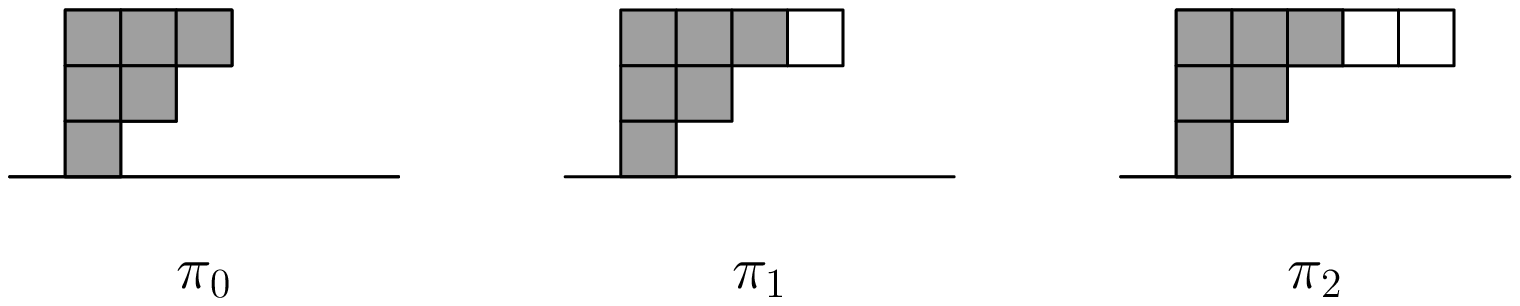}
\end{center}
If we set $\sig_i:=\sig(\pi_i)$, $\tau_i:=\tau(\pi_i)$, then $\sig_0, \tau_0$ are the respective multiplicative identities of the Chow rings, while $\sig_1,\sig_2$ (resp. $\tau_1,\tau_2$), called {\em special} cycles, are the nontrivial Chern classes of the tautological bundle over $X_B$ (resp. $X_C$).  Pragacz and Ratajski prove that $\tau_1,\tau_2$ algebraically generate $\CH(X_C)$ (Cor. 1.8 and Lem. 3.2), while $\sig_1,\sig_2$ only generate $\CH(X_B)$ after tensoring with $\Z[1/2]$ (Thm. 10.1).

With the weight $|\lam|$ of a shape defined as above, we have $\sig(\lam)\in\CH^{\abs{\lam}}(X_B)$ and $\tau(\lam)\in\CH^{\abs{\lam}}(X_C)$.  In particular, $\codim(\sig_i)=\abs{\pi_i}=i$ and $\codim(\tau_i)=\abs{\pi_i}=i$.

The multiplication rules in $\CH(X_B)$ and $\CH(X_C)$ are very similar, differing only by some factors of $2$ in certain multiplicities.  Indeed, for any shape $\lam\in\Pt$, $i=1,2$, we have the Pieri-type formulas $$\sig(\lam)\cdot\sig_i = \sum 2^{e_B(\lam,\mu)}\sig(\mu)$$ $$\tau(\lam)\cdot\tau_i = \sum 2^{e_C(\lam,\mu)}\tau(\mu)$$ for multiplying a basic cycle by a special cycle \cite[Thms. 2.2 and 10.1]{PR96}.  Here, the sums are over all $\mu$ compatible with $\lambda$ satisfying $\abs{\mu}=\abs{\lam}+i$, and $e_B(\lam,\mu), e_C(\lam,\mu)$ are the cardinalities of certain sets of components of the skew diagram $D_\mu^b\bs D_\lam^b$.  For our purposes, what we need is that for compatible $\lam,\mu$ with $\abs{\mu}=\abs{\lam}+i$, the difference $e_B(\lam,\mu)-e_C(\lam,\mu)$ equals the number of {\em extremal} components of $D_\mu^b\bs D_\lam^b$.  (This follows from the fact that an extremal component is not related and has no $(\mu-\lam)$-boxes lying over it, by parts 2 and 4 of the definition \cite[Def. 2.1]{PR96} of compatible shapes.)  The skew diagram $D_\mu^b\bs D_\lam^b$ clearly has at most one extremal component.  There is exactly one extremal component if and only if $l(\mu^b)>l(\lam^b)$, which by part 5 of the definition of compatible shapes is equivalent to $l(\mu^b) = l(\lam^b)+1$.  There are no extremal components if and only if $l(\mu^b) = l(\lam^b)$.  Putting all of this together, we conclude that
\begin{equation}\label{mult}
e_B(\lam,\mu)-e_C(\lam,\mu) = l(\mu^b)-l(\lam^b)\in\{0,1\}.
\end{equation}

To describe the product of several special cycles, we need to extend the notion of compatibility to sequences of shapes.  For nonnegative integers $a_1,a_2$, define a {\em compatible $(a_1,a_2)$-chain} to be a sequence of shapes $$\Lam = (\pi_0\!=\!\lam_0, \lam_1, \ldots, \lam_{a_1+a_2})$$ such that for $i=1,2,\ldots,(a_1\!+\!a_2)$, the shapes $\lam_i$ and $\lam_{i-1}$ are compatible, $\abs{\lam_i}-\abs{\lam_{i-1}}\in\{1,2\}$, and $\abs{\lam_{a_1+a_2}}=a_1+2a_2$.

We now can write down formulas for an arbitrary product of special cycles in $\CH(X_B)$ or $\CH(X_C)$.  In $\CH(X_B)$, the formula is $$\sig_1^{a_1}\cdot\sig_2^{a_2} = \sum_{\subs{\text{\tiny compatible $(a_1,a_2)$-chains}}{\Lam=(\lam_0, \lam_1, \ldots, \lam_{a_1+a_2})}}2^{b_\Lam}\sig(\lam_{a_1+a_2}),$$ where $$b_\Lam = e_{B}(\lam_0,\lam_1)+e_{B}(\lam_1,\lam_2)+\cdots+e_{B}(\lam_{a_1+a_2-1},\lam_{a_1+a_2}),$$ and in $\CH(X_C)$, the formula is $$\tau_1^{a_1}\cdot\tau_2^{a_2} = \sum_{\subs{\text{\tiny compatible $(a_1,a_2)$-chains}}{\Lam=(\lam_0, \lam_1, \ldots, \lam_{a_1+a_2})}}2^{c_\Lam}\tau(\lam_{a_1+a_2}),$$ where $$c_\Lam = e_{C}(\lam_0,\lam_1)+e_{C}(\lam_1,\lam_2)+\cdots+e_{C}(\lam_{a_1+a_2-1},\lam_{a_1+a_2}).$$
It follows from equation (\ref{mult}) that
\begin{eqnarray}\label{teles} b_\Lam-c_\Lam & = & [l(\lam_1^b)-l(\lam_0^b)]+[l(\lam_2^b)-l(\lam_1^b)]+\cdots+[l(\lam_{a_1+a_2}^b)-l(\lam_{a_1+a_2-1}^b)] \nonumber \\
& = & l(\lam^b_{a_1+a_2})-l(\emptyset)\\
& = & l(\lam^b_{a_1+a_2}). \nonumber
\end{eqnarray}

We now are ready to prove the lemma.

\begin{proof}
It is enough to take $\gam$ to be a basic cycle, say $\gam=\sig(\lam)$ for some shape $\lam$ of weight $r$.  Since $\tau_1$ and $\tau_2$ generate the ring $\CH(X_C)$, there exist integers $u_j$ such that
$$\tau(\lam) = \sum_{j=0}^{\lfloor r/2\rfloor} u_j \left(\tau_1^{r-2j}\cdot\tau_2^{j}\right)\in\CH^r(X_C).$$
Define $$\gam':= \sum_{j=0}^{\lfloor r/2\rfloor} u_j \left(\sig_1^{r-2j}\cdot\sig_2^{j}\right)\in\CH^r(X_B).$$  This is a rational cycle, since $\sig_1,\sig_2$ are Chern classes of the tautological bundle over $X_B=\bar{X}_2$ and are therefore defined over the base field $F$.

It remains to show that $\gam'=2\gam$.  The key is that for any shape $\lam=(\lam^t//\lam^b)$ of weight $2n-3$ or $2n-2$, $l(\lam^b)=1$.  Indeed, it follows easily from the conditions imposed in the definition of a shape $\lam$ that $l(\lam^b)=0$ implies $\abs{\lam}\leq 2n-4$, and $l(\lam^b)=2$ (the greatest value possible) implies $\abs{\lam}\geq 2n-1$.  The case $n=5$ is illustrated below, where we shade the ``$\pi_0$ boxes'' which don't contribute to the weight $|\lam|$.  Diagram I corresponds to the shape of maximal weight $2n-4=6$ among shapes $\lam$ with $l(\lam^b)=0$.  Diagram II corresponds to the shape of minimal weight $2n-1=9$ among shapes $\lam$ with $l(\lam^b)=2$.
\begin{center}
\includegraphics[scale=0.75]{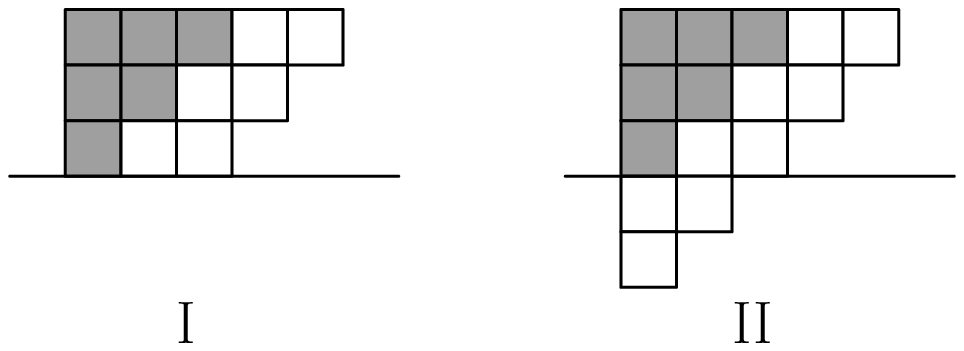}
\end{center}

Expanding the products in the expression for $\tau(\lam)$, we get
$$\tau(\lam) = \sum_{j=0}^{\lfloor r/2\rfloor} u_j \sum_{\subs{\text{\tiny compatible $(r\!-\!2j,j)$-chains}}{\Lam=(\lam_0, \lam_1, \ldots, \lam_{r-j})}}2^{c_\Lam}\tau(\lam_{r-j}).$$
The maps $\sig$ and $\tau$ induce a group isomorphism $\mbox{``}\sig\circ\tau^{-1}\mbox{''}:\CH(X_C)\ra\CH(X_B)$ which when applied to the equation above yields
\begin{equation}\label{tau}\sig(\lam) = \sum_{j=0}^{\lfloor r/2\rfloor} u_j \sum_{\subs{\text{\tiny compatible $(r\!-\!2j,j)$-chains}}{\Lam=(\lam_0, \lam_1, \ldots, \lam_{r-j})}}2^{c_\Lam}\sig(\lam_{r-j}).
\end{equation}

On the other hand, expanding the expression for $\gam'$ yields
\begin{equation}\label{gamp}\gam' = \sum_{j=0}^{\lfloor r/2\rfloor} u_j \sum_{\subs{\text{\tiny compatible $(r\!-\!2j,j)$-chains}}{\Lam=(\lam_0, \lam_1, \ldots, \lam_{r-j})}}2^{b_\Lam}\sig(\lam_{r-j}),
\end{equation}
which differs from the expression for $\sig(\lam)$ only in that the exponent $c_\Lam$ has been changed to $b_\Lam$.  By equation (\ref{teles}) and our length computation above, $b_\Lam-c_\Lam=l(\lam^b_{r-j})=1$ for any compatible $(r\!-\!2j,j)$-chain $\Lam$, since $$\abs{\lam_{r-j}}=(r-2j)+2j=r\in\{2n-3,2n-2\}.$$  Thus each term on the right-hand side of (\ref{gamp}) is twice the corresponding term on the right-hand side of (\ref{tau}) and $\gam'=2\sig(\lam)=2\gam$.
\end{proof}

\section{Proof of $2$-incompressibility}

We now have the ingredients necessary for a proof of the main theorem, whose statement we repeat below.

\begin{theorem}
If $\deg\CH(X_2)=4\Z$ and $\ki_2(\phi)=1$, then $X_2$ is $2$-in\-com\-pres\-sible.  In particular, $$\cdim_2(X_2)=\cdim(X_2)=\dim(X_2)=4n-5.$$
\end{theorem}

We briefly recall some terminology from \cite[\S 62 and \S 75]{EKM08}. Let $X$ and $Y$ be schemes with $\dim X=e$.  A \emph{correspondence of degree zero $\de:X\lt Y$ from $X$ to $Y$} is just a cycle $\de\in\CH_e(X\times Y)$.  The \emph{multiplicity} $\mult(\de)$ of such a $\de$ is the integer satisfying $\mult(\de)\cdot [X] = p_*(\de)$, where $p_*$ is the push-forward homomorphism
$$p_*:\CH_e(X\times Y)\ra \CH_e(X) = \mathbb{Z}\cdot [X].$$  The exchange isomorphism $X\times Y\ra Y\times X$ induces an isomorphism
$$\CH_e(X\times Y)\ra \CH_e(Y\times X)$$
sending a cycle $\de$ to its \emph{transpose} $\de^t$.

\begin{proof}
To prove that a variety $X$ is $2$-incompressible, it suffices to show that for any correspondence $\de:X\lt X$ of degree zero,
\begin{equation}\label{mod}
\mult(\de)\equiv\mult(\de^t) \pmod{2}.
\end{equation}
Indeed, suppose we have $f:X'\ra X$ and a dominant $g:X'\ra X$ with $F(X')/F(X)$ finite of odd degree.  Let $\de\in\CH(X\times X)$ be the pushforward of the class $[X']$ along the induced morphism $(g,f):X'\ra X\times X$.  By assumption, $\mult(\de)$ is odd, so by (\ref{mod}) we have that $\mult(\de^t)$ is odd.  It follows that $f_*([X'])$ is an odd multiple of $[X]$ and in particular is nonzero, so $f$ is dominant.

We will check that the condition (\ref{mod}) holds for the variety $X_2$.  A correspondence of degree zero $\de:X_2\lt X_2$ is just an element of $\CH_{4n-5}(X_2\times X_2)$.  Using the method of Chernousov and Merkurjev described in \cite{CM06}, we can decompose the motive of $X_2\times X_2$ as follows.  See also \cite{CSM05} for examples of similar computations.

We first realize $X_2$ as a projective homogeneous variety.  Let $G$ denote the special orthogonal group corresponding to the quadratic form $\phi$ on $V$.  Let $\Pi$ be a set of simple roots for the root system $\Sigma$ of $G$.  If $e_1,\ldots,e_n$ are the standard basis vectors of $\mathbb{R}^n$, we may take $$\Pi=\{\al_1\!:=\!e_1\!-\!e_2,\ldots,\al_{n-1}\!:= \!e_{n-1}\!-\!e_n,\al_n\!:=\!e_n\}.$$  Then $X_2$ is a projective $G$-homogeneous variety, namely the variety of all parabolic subgroups of $G$ of type $S$, for the subset $S=\Pi\backslash\{\al_2\}$ of the set of simple roots.

Let $W$ denote the Weyl group of the root system $\Sigma$.  When $n\geq 4$, there are six double cosets $D\in W_P\backslash W/W_P$ with representatives $w$ listed in the first column below (where we write $\al_k$ when we mean the reflection $w_{\al_k}$).  The second column lists the effect of $w^{-1}$ on the first four $e_i$ (the rest not being affected).  The third column gives the subset of $\Pi$ associated to $w$.  When $n=3$, there are only five double cosets.  In this case, the table may be amended by deleting the final row and removing all mention of $e_4$ from the remaining rows.
\[
\begin{array}{|c|c|l|} \hline
\rule[-2.5mm]{0mm}{8mm} \boldsymbol{w}\text{ ($\al_k$ here means $w_{\al_k}$)} & \boldsymbol{w^{-1}(e_1,e_2,e_3,e_4)} & \qquad\boldsymbol{R_D} \\ \hline
\rule[0mm]{0mm}{4mm} 1 & (e_1,e_2,e_3,e_4) & \Pi\backslash \{\al_2\} \\
\al_2\cdots\al_n\cdots\al_2 & (e_1,-e_2,e_3,e_4) & \Pi\backslash \{\al_1,\al_2\} \\
(\al_2\cdots\al_n\cdots\al_2)\al_1(\al_2\cdots\al_n\cdots\al_2) & (-e_2,-e_1,e_3,e_4) & \Pi\backslash \{\al_2\} \\
\al_1 & (e_1,e_3,e_2,e_4) & \Pi\backslash \{\al_1,\al_2,\al_3\} \\
\al_2\al_1(\al_3\cdots\al_n\cdots\al_2) & (e_3,-e_2,e_1,e_4) & \Pi\backslash \{\al_1,\al_2,\al_3\} \\
\rule[-2mm]{0mm}{0mm}(\al_2\al_1)(\al_3\al_2) & (e_3,e_4,e_1,e_2) & \Pi\backslash \{\al_2,\al_4\} \\ \hline
\end{array}
\]

From \cite[Thm. 6.3]{CM06}, we deduce the following decomposition of the motive of $X_2\times X_2$, where the last summand is removed for the case $n=3$.  We denote by $X_{d_1,d_2,\ldots,d_s}$ the variety of flags of totally isotropic subspaces of $V$ of dimensions $d_1,d_2,\ldots,d_s$.

\[
\begin{split}
M(X_2\times X_2) \simeq & \;M(X_2) \ds M(X_{1,2})(2n-3) \ds M(X_2)(4n-5) \\
  &\ds M(X_{1,2,3})(1) \ds M(X_{1,2,3})(2n-2) \ds \Big[M(X_{2,4})(4)\Big]
\end{split}
\]

This in turn yields a decomposition of the middle-dimensional component of the Chow group of $X_2\times X_2$.

\[
\begin{split}
\CH_{4n-5}(X_2\times X_2) \simeq & \,\CH_{4n-5}(X_2) \ds \CH_{2n-2}(X_{1,2}) \ds \CH_0(X_2) \\
 & \ds \CH_{4n-6}(X_{1,2,3}) \ds \CH_{2n-3}(X_{1,2,3}) \ds \Big[\CH_{4n-9}(X_{2,4})\Big]
\end{split}
\]

It now suffices to check the congruence $\mult(\de)\equiv\mult(\de^t) \pmod{2}$ for $\de$ in the image of any of these summands.  The embedding of the first summand $\CH_{4n-5}(X_2)$ is induced by the diagonal morphism $X_2\ra X_2\times X_2$, so the multiplicities are equal by symmetry.

Any element $\de$ of the third summand $\CH_0(X_2)$ has degree divisible by $4$ by assumption, hence its image in the Chow group $\CH_0(\bar{X}_2)$ is divisible by $4$.  (Here we use that $\CH_0(\bar{X}_2)$ is generated by a single element of degree $1$.)  The image of $\de$ in $\CH_{4n-5}(\bar{X}_2\times \bar{X}_2)$ is then also divisible by $4$, and since multiplicity does not change under field extension, $\mult(\de)\equiv 0\equiv \mult(\de^t) \pmod{4}$.

The second summand requires our work from the previous section.  Since $X_{1,2}$ is a projective bundle over $X_2$, there is a motivic decomposition $M(X_{1,2})\simeq M(X_2)\ds M(X_2)(1)$, so that $$\CH_{2n-2}(X_{1,2})\simeq \CH_{2n-2}(X_2)\ds \CH_{2n-3}(X_2).$$  It is enough to consider $\de$ equal to the image of some $\beta\in\CH_r(X_2)$, where $r\in\{2n-3,2n-2\}$.  By the same reasoning as in the previous paragraph, it suffices to show that the image of $\beta$ in $\ol{\CH}_r(X_2)\subset\CH_r(\bar{X}_2)$ is divisible by $2$ in $\CH_r(\bar{X}_2)$.  Suppose it is not.  Then the image $\hat{\beta}$ of $\beta$ in the modulo-$2$ Chow group $$\Ch_r(\bar{X}_2):=\CH_r(\bar{X}_2)/2\CH_r(\bar{X}_2)$$ is nonzero.  By \cite[Rem. 5.6]{KM06}, the ``cellular'' variety $\bar{X}_2$ is ``$2$-balanced,'' i.e. the bilinear form $(\hat{\beta},\hat{\gamma})\mapsto \deg (\hat{\beta}\cdot\hat{\gamma})$ on $\Ch(\bar{X}_2)$ is nondegenerate.  Hence there exists $\gamma\in\CH^r(\bar{X}_2)$ such that $$\deg(\beta\cdot\gamma)\equiv 1\pmod{2}.$$  Since $2\gamma$ is rational by our Lemma \ref{tech}, we have $$\deg\ol{\CH}_0(X_2) \ni \deg(\beta\cdot 2\gamma)\equiv 2\pmod{4}.$$  Degree does not change under field extension, so this contradicts our assumption that $\deg\CH(X_2)=4\Z$.

The last three summands of the decomposition are dealt with by the following proposition, whose proof uses our results on higher Witt indices.  This will complete the proof of the theorem.
\end{proof}

\begin{proposition}
Let $Fl:=X_{d_1,d_2,\ldots,d_s}$ be a variety of totally isotropic flags with $d_s>2$ and let the correspondence $\al:Fl\lt X_2\times X_2$ induce an embedding $$\al_*:\CH_r(Fl)\hookrightarrow \CH_{4n-5}(X_2\times X_2).$$  Then for any $\delta$ in the image of $\al_*$, $\mult(\de)\equiv 0\equiv \mult(\de^t) \pmod{2}$.
\end{proposition}

\begin{proof}
Consider the diagram below of fiber products, where we select either of the projections $p_i$ and choose the other morphisms accordingly.

\[
\xymatrix@!C{
                                        & (Fl)_{F(X_2)}  \ar[dr]                       &   \\
(Fl\times X_2)_{F(X_2)} \ar[r] \ar[d] \ar[ur]  & (X_2)_{F(X_2)} \ar[r]  \ar[d]      & \Spec F(X_2) \ar[d]    \\
Fl\times X_2\times X_2 \ar[r] \ar[d]    & X_2\times X_2 \ar[r]_-{p_2}^-{p_1}    & X_2 \\
Fl
}
\]

Taking push-forwards and pull-backs, we get the following diagram which commutes except for the triangle at the bottom.  The push-forward by $p_i$ takes a cycle $\de\in\CH_{4n-5}(X_2\times X_2)$ to $\mult(\de)$ if we chose the first projection $p_1$ or to $\mult(\de^t)$ if we chose the second projection $p_2$.

\[
\xymatrix@!C{
      & \CH_0\left((Fl)_{F(X_2)}\right)  \ar[dr]+<-1.5ex,1ex>^-{\deg}                &   \\
\CH_0\left((Fl\times X_2)_{F(X_2)}\right) \ar[r] \ar[ur]   & \CH_0\left((X_2)_{F(X_2)}\right) \ar[r]_-{\deg}  & \Z \ar@{=}[d] \\
\CH_{4n-5}\left(Fl\times X_2\times X_2\right) \ar[r] \ar[u]    & \CH_{4n-5}\left(X_2\times X_2\right) \ar[r]_-{(\mult)\circ(\trans)}^-{\mult} \ar[u]    & \Z                     \\
\CH_r(Fl) \ar@/_1pc/_-{\al_*}@{.>}[ur] \ar[u]
}
\]

Any $\de\in \im(\al_*)$ also lies in the image of $\CH_{4n-5}\left(Fl\times X_2\times X_2\right)$, by the definition of the push-forward.  Chasing through the diagram, one sees that $\mult(\de)$ (and similarly $\mult(\de^t)$) must lie in $\deg \CH_0\left((Fl)_{F(X_2)}\right)$.  By Proposition \ref{witt1}, we know that $\kj_2(\phi)=2$.  Hence by Proposition \ref{witt2}, $\phi_{F(X_2)}\simeq 2\H\perp\psi$ for some anisotropic quadratic form $\psi$ over $F(X_2)$.  In order for the variety $Fl_{F(X_2)}$ to have a rational point over a field extension $K$ of $F(X_2)$, $\psi$ must be isotropic over $K$, due to the assumption $d_s>2$.  By Springer's Theorem, if the degree of the extension $K$ is finite then it must be divisible by 2, so $$\deg \CH_0\left((Fl)_{F(X_2)}\right)\subset 2\Z.$$
\end{proof}

This completes the proof of the theorem.

\end{document}